\documentclass{amsproc}
\usepackage{amssymb}
\usepackage[all]{xy}

\date{20 August 2009}

\newtheorem{thm}{Theorem}
\newtheorem{prop}[thm]{Proposition}
\newtheorem{cor}[thm]{Corollary}

\theoremstyle{remark}
\newtheorem{example}[thm]{Example}
\newtheorem{remark}[thm]{Remark}

\newcommand\cs{\ensuremath{C^{*}}}
\newcommand\set[1]{\{\,#1\,\}}

\newcommand\comcsunital{\ensuremath{\operatorname{\mathtt{CommC*1}}}}
\newcommand\topcpt{\ensuremath{\operatorname{\mathtt{Cpct}}}}

\newcommand\topcts{\ensuremath{\operatorname{\mathtt{LCpct}}}} 
\newcommand\cscatnd{%
\ensuremath{\operatorname{\mathtt{C*}}_{\!\!\!\!\mathtt{nd}}}}
\newcommand\comcscatnd{%
\ensuremath{\operatorname{\mathtt{CommC*}}_{\!\!\!\!\mathtt{nd}}}}

\newcommand\cscoactndland{%
\ensuremath{\operatorname{\mathtt{C*coact}}_{\mathtt{nd}}^{\mathtt{n}}}}
\newcommand\cscoactnd{%
\ensuremath{\operatorname{\mathtt{C*coact}}_{\mathtt{nd}}}}

\newcommand\csactnd{%
\ensuremath{\operatorname{\mathtt{C*act}}_{\mathtt{nd}}}}

\newcommand\cscat{\ensuremath{\operatorname{\mathtt{C*}}}}
\newcommand\csact{\ensuremath{\operatorname{\mathtt{C*act}}}}
\newcommand\csactgct{\ensuremath{\cscat(G,(C_{0}(T),\rt))}}
\newcommand\cscoactn{\operatorname{\mathtt{C*coact}}^{\mathtt{n}}}
\newcommand\CP{\operatorname{\mathtt{CP}}}
\newcommand\Fix{\operatorname{\mathtt{Fix}}}

\newcommand\RCP{\operatorname{\mathtt{RCP}}}
\newcommand\RCPgh{\operatorname{\mathtt{RCP}}_{\mathtt{G/H}}}
\newcommand\RCPh{\operatorname{\mathtt{RCP}}_{\mathtt{H}}}
\DeclareMathOperator{\Aut}{Aut}
\newcommand\cb{C_{b}}
\newcommand\tensor{\otimes}
\DeclareMathOperator{\id}{id}
\DeclareMathOperator{\rt}{rt}

\DeclareMathOperator{\supp}{supp}
\newcommand\Z{\mathbb{Z}}
\newcommand\R{\mathbb{R}}

\newcommand\C{\mathbb{C}}
\renewcommand\L{\mathcal{L}}
\newcommand\K{\mathcal{K}}
\newcommand\HH{\mathcal{H}}
\DeclareMathOperator{\Ind}{Ind}
\DeclareMathOperator{\Ad}{Ad}
\DeclareMathOperator{\clsp}{\overline{span}}

\DeclareMathOperator{\Mor}{Mor}
\DeclareMathOperator{\Obj}{Obj}
\def\dashind{\operatorname{\!-Ind}}
\let\phi\varphi

\def\<{\langle}
\def\>{\rangle}
\let\ipscriptstyle=\scriptscriptstyle
\def\lipsqueeze{{\mskip -3.0mu}}
\def\ripsqueeze{{\mskip -3.0mu}}
\def\ipcomma{\nobreak\mathrel{,}\nobreak}
\newbox\ipstrutbox
\setbox\ipstrutbox=\hbox{\vrule height8.5pt
width 0pt}
\def\ipstrut{\copy\ipstrutbox}
\def\lip#1<#2,#3>{\mathopen{\relax_{\ipstrut\ipscriptstyle{
#1}}\lipsqueeze
\langle} #2\ipcomma #3 \rangle}
\def\blip#1<#2,#3>{\mathopen{\relax_{\ipstrut
\ipscriptstyle{ #1}}\lipsqueeze\bigl\langle} #2\ipcomma #3 \bigr\rangle}
\def\rip#1<#2,#3>{\langle #2\ipcomma #3
\rangle_{\ripsqueeze\ipstrut\ipscriptstyle{#1}}}
\def\brip#1<#2,#3>{\bigl\langle #2\ipcomma #3
\bigr\rangle_{\ripsqueeze\ipstrut\ipscriptstyle{#1}}}
\def\angsqueeze{\mskip -6mu}
\def\smangsqueeze{\mskip -3.7mu}
\def\trip#1<#2,#3>{\langle\smangsqueeze\langle #2\ipcomma #3
\rangle\smangsqueeze\rangle_{\ripsqueeze\ipstrut\ipscriptstyle{#1}}}
\def\btrip#1<#2,#3>{\bigl\langle\angsqueeze\bigl\langle #2\ipcomma
#3
\bigr\rangle
\angsqueeze\bigr\rangle_{\ripsqueeze\ipstrut\ipscriptstyle{#1}}}
\def\tlip#1<#2,#3>{\mathopen{\relax_{\ipstrut\ipscriptstyle{
#1}}\lipsqueeze \langle\smangsqueeze\langle} #2\ipcomma #3
\rangle\smangsqueeze\rangle}
\def\btlip#1<#2,#3>{\mathopen{\relax_{\ipstrut\ipscriptstyle{
#1}}\lipsqueeze
\bigl\langle\angsqueeze\bigl\langle} #2\ipcomma #3
\bigr\rangle\angsqueeze\bigr\rangle}

\def\ip(#1|#2){(#1\mid #2)}
\def\bip(#1|#2){\bigl(#1 \mid #2\bigr)}
\def\Bip(#1|#2){\Bigl( #1 \bigm| #2 \Bigr)}

\begin{document}

\title[Proper Actions]{
Functoriality of Rieffel's Generalised
Fixed-Point\\ Algebras for Proper Actions}

\author[an Huef]{Astrid an Huef}
\address{School of Mathematics and Statistics\\
The University of New South Wales\\Sydney\\
NSW 2052\\
Australia}
\email{astrid@unsw.edu.au}


\author[Raeburn]{Iain Raeburn}
\address{School  of Mathematics and Applied Statistics, University of
Wollongong, NSW 2522, Australia}
\email{raeburn@uow.edu.au}

\author[Williams]{Dana P. Williams}
\address{Department of Mathematics\\Dartmouth College\\
Hanover, NH 03755\\USA}
\email{dana.williams@dartmouth.edu}

\begin{abstract}
  We consider two categories of $C^*$-algebras; in the first, the
  isomorphisms are ordinary isomorphisms, and in the second, the
  isomorphisms are Morita equivalences. We show how these two
  categories, and categories of dynamical systems based on them, crop
  up in a variety of $C^*$-algebraic contexts. We show that Rieffel's
  construction of a fixed-point algebra for a proper action can be
  made into functors defined on these categories, and that his Morita
  equivalence then gives a natural isomorphism between these functors
  and crossed-product functors. These results have interesting
  applications to non-abelian duality for crossed products.
\end{abstract}

\thanks{This research was supported by the Australian Research Council
  and the Edward Shapiro fund at Dartmouth College.}

\date{\today}
\maketitle


\section*{Introduction}

Let $\alpha$ be an action of a locally compact group $G$ on a
$C^*$-algebra $A$. In \cite{proper}, Rieffel studied a class of proper
actions for which there is a Morita equivalence between the reduced
crossed product $A\rtimes_{\alpha,r}G$ and a generalised fixed-point
algebra $A^\alpha$ sitting inside the multiplier algebra
$M(A)$. Rieffel subsequently proved that $\alpha$ is proper whenever
there is a free and proper $G$-space $T$ and an equivariant embedding
$\phi:C_0(T)\to M(A)$ \cite[Theorem~5.7]{integrable}. In
\cite{kqrproper}, inspired by previous work of Kaliszewski and Quigg
\cite{kqcat}, it was observed that Rieffel's hypothesis says precisely
that $((A,\alpha),\phi)$ is an object in a comma category of dynamical
systems. It therefore becomes possible to ask questions about the
functoriality of Rieffel's construction, and about the naturality of
his Morita equivalence.

These questions have been tackled in several recent papers
\cite{kqrproper,aHKRWproper-n, aHKRW4}, which we believe contain some
very interesting results. In particular, they have substantial
applications to non-abelian duality for $C^*$-algebraic dynamical
systems. However, these papers also contain a confusing array of
categories and functors. So our goal here is to discuss the main
categories and explain why people are interested in them. We will then
review some of the main results of the papers \cite{kqcat,
  kqrproper,aHKRWproper-n, aHKRW4}, and try to explain why we find
them interesting.

In all the categories of interest to us, the objects are either
$C^*$-algebras or dynamical systems involving actions or coactions of
a fixed group on $C^*$-algebras. But when we decide what morphisms to
use, we have to make a choice, and what we choose depends on what sort
of theorems we are interested in. Loosely speaking, we have to decide
whether we want the isomorphisms in our category to be the usual
isomorphisms of $C^*$-algebras, or to be Morita equivalences. We think
that, once we have made that decision, there is a ``correct'' way to
go forward.

We begin in \S\ref{sec:introduction} with a discussion of commutative
$C^*$-algebras; since Morita equivalence does not preserve
commutativity, it is clear that in this case we want isomorphisms to
be the usual isomorphisms. However, even then we have to
do something a little odd: we want the morphisms from $A$ to $B$ to be
homomorphisms $\phi:A\to M(B)$. Once we have the right category, we
can see that operator algebraists have been implicitly working in this
category for years. The motivating example for Kaliszewski and Quigg
was a duality theory for dynamical systems due to Landstad \cite{L},
and our main motivation is, as we said above, to understand Rieffel's
proper actions. We discuss Landstad duality in \S\ref{sec:Ld}. In
\S\ref{sec:QLd}, we discuss its analogue for crossed products by
coactions, which is due to Quigg \cite{Q}, and how this makes contact
with Rieffel's theory of proper actions.

We begin \S\ref{sec:naturality} by showing how the search for
naturality results leads us to a different category $\cscat$ of
$C^*$-algebras in which the morphisms are based on right-Hilbert
bimodules. Categories of this kind have been round much longer, and
\cite{taco,enchilada}, for example, contain a detailed discussion of
how imprimitivity theorems provide natural isomorphisms between
functors with values in $\cscat$. In \S\ref{sec:upgrading-cscat-its},
we discuss a theorem from \cite{aHKRWproper-n} which says that
Rieffel's Morita equivalences give a natural isomorphism between a
crossed-product functor and a fixed-point-algebra functor. This
powerful result implies, for example, that the version in \cite{hrman}
of Mansfield imprimitivity for arbitrary subgroups is natural. We
finish with a brief survey of one of the main results of \cite{aHKRW4}
which uses an approach based on Rieffel's theory to establish
induction-in-stages for crossed products by coactions.

\section{The category $\cscatnd$ and 
commutative $C^*$-algebras} 
\label{sec:introduction}

In our first course in $C^*$-algebras, we learned that commutative
unital $C^*$-algebras are basically the same things as compact
topological spaces.  To make this formal, we note that the assignment
$X\mapsto C(X)$ is the object map in a contravariant functor $C$ from
the category $\topcpt$ of compact Hausdorff spaces and continuous
functions to the category $\comcsunital$ of unital commutative
$C^*$-algebras and unital homomorphisms (which for us are always
$*$-preserving); the morphism $C(f)$ associated to a continuous map
$f:X\to Y$ sends $a\in C(Y)$ to $a\circ f\in C(X)$. Then the
Gelfand-Naimark theorem implies that the functor $C$ is an equivalence
of categories. (This result goes back to \cite{N}, and we will go into
the details of what it means in Theorem~\ref{nondegcommOK} below.)

The Gelfand-Naimark theorem for non-unital algebras says that
commutative $C^*$-algebras are basically the same things as locally
compact topological spaces. However, it is not so easy to put this
version in a categorical context, and in doing so we run into some
important issues which are very relevant to problems involving crossed
products and non-abelian duality. So we will discuss these issues now
as motivation for our later choices.

There is no doubt what the analogue of the functor $C$ does to
objects: it takes a locally compact Hausdorff space $X$ to the
$C^*$-algebra $C_0(X)$ of continuous functions $a:X\to \C$ which
vanish at infinity. However, there is a problem with morphisms:
composing with a continuous function $f:X\to Y$ does not necessarily
map $C_0(Y)$ into $C_0(X)$. For example, consider the function
$f:\R\to \R$ defined by $f(x)=(1+x^2)^{-1}$: any function $a\in
C_0(\R)$ which is identically $1$ on $[0,1]$ satisfies $a\circ f=1$,
and hence $a\circ f$ does not vanish at infinity. One way out is to
restrict attention to the category in which the morphisms from $X$ to
$Y$ are the \emph{proper functions} $f:X\to Y$ for which inverse
images of compact sets are compact, and then on the $C^*$-algebra side
one has to restrict attention to the homomorphisms $\phi:A\to B$ such
that the products $\phi(a)b$ span a dense subspace of $B$. In
\cite{P}, Pedersen does exactly this, and calls these \emph{proper
  homomorphisms}. It turns out, though, that there is a very
satisfactory way to handle arbitrary continuous functions between
locally compact spaces, in which we allow morphisms which take values in $C_b(X)$.

A homomorphism $\phi$ of one $C^*$-algebra $A$ into the multiplier
algebra $M(B)$ of another $C^*$-algebra $B$ is called
\emph{nondegenerate} if $\phi(A)B:=\clsp\{\phi(a)b\}$ is all of
$B$. (This notation is suggestive: the Cohen factorisation theorem  says that
everything in the closed span factors as $\phi(a)b$.) We want to think of the nondegenerate
homomorphisms $\phi:A\to M(B)$ as morphisms from $A$ to $B$. Every
nondegenerate homomorphism $\phi$ extends to a unital homomorphism
$\bar \phi:M(A)\to M(B)$ (see \cite[Corollary~2.51]{tfb}, for
example); the extension has to satisfy
$\bar\phi(m)(\phi(a)b)=\phi(ma)b$, and hence the non-degeneracy
implies that there is exactly one such extension, and that it is
strictly continuous.

The following fundamental proposition is implicit in
\cite[\S1]{kqcat}.

\begin{prop}
  There is a category $\cscatnd$ in which the objects are
  $C^*$-algebras, the morphisms from $A$ to $B$ are the nondegenerate
  homomorphisms from $A$ to $M(B)$, and the composition of $\phi:A\to
  M(B)$ and $\psi:B\to M(C)$ is $\psi\circ\phi:=\bar\psi\circ
  \phi$. The isomorphisms in this category are the usual isomorphisms
  of $C^*$-algebras.
\end{prop}

\begin{proof}
  It is easy to check that the composition $\bar\psi\circ \phi:A\to
  M(C)$ is nondegenerate, and hence defines a morphism in
  $\cscatnd$. Since $\bar\psi\circ\bar\phi$ is a homomorphism from
  $M(A)$ to $M(C)$ which extends $\bar\psi\circ\phi$, it must be the
  unique extension $\overline{\bar\psi\circ \phi}$. Thus if
  $\theta:C\to M(D)$ is another nondegenerate homomorphism, we have
  \begin{align*}
    \theta\circ(\psi\circ\phi)&= \bar\theta\circ(\psi\circ\phi)=
    \bar\theta\circ(\bar\psi\circ\phi) 
    =(\bar\theta\circ\bar\psi)\circ\phi\\
    &=(\overline{\bar\theta\circ
      \psi})\circ\phi=(\overline{\theta\circ\psi})
    \circ\phi=(\theta\circ\psi)\circ\phi, 
  \end{align*}
  and composition in $\cscatnd$ is associative. The identity maps
  $\id_A:A\to A$, viewed as homomorphisms into $M(A)$, satisfy
  $\bar\id_A=\id_{M(A)}$, and hence have the properties one requires
  of the identity morphisms in $\cscatnd$. Thus $\cscatnd$ is a
  category, as claimed.

  For the last comment, notice first that every isomorphism is
  trivially nondegenerate, and hence defines a morphism in $\cscatnd$,
  which is an isomorphism because it has an inverse. Conversely,
  suppose that $\phi:A\to M(B)$ and $\psi:B\to M(A)$ are inverses of
  each other in $\cscatnd$, so that $\bar\psi\circ \phi=\id_A$ and
  $\bar\phi\circ \psi=\id_B$. Using first the non-degeneracy of $\psi$
  and then the non-degeneracy of $\phi$, we obtain
  \[
  \phi(A)=\phi(\psi(B)A)=\bar\phi(\psi(B))\phi(A)=B\phi(A)=B.
  \]
  Thus $\phi$ has range $B$, and since $\bar\psi|_B=\psi$, we have
  $\psi\circ \phi=\id_A$. The same arguments show that $\phi\circ
  \psi=\id_B$, so $\phi$ is an isomorphism in the usual sense.
\end{proof}


If $f:X\to Y$ is a continuous map between locally compact spaces and
$a\in C_0(Y)$, then $a\circ f$ is a continuous bounded function which
defines a multiplier of $C_0(X)$. For every $b$ in the dense
subalgebra $C_c(X)$, we can choose $a\in C_c(Y)$ such that $a=1$ on
$f(\supp b)$, and then $b=(a\circ f)b$, so $C_0(f):a\mapsto a\circ f$
is a nondegenerate homomorphism from $C_0(Y)$ to $M(C_0(X))$; the
extension $\overline{C_0(f)}$ to $C_b(X)=M(C_0(X))$ is again given by
composition with $f$. We now have a functor $C_0$ from the category
$\topcts$ of locally compact spaces and continuous maps to the full
subcategory $\comcscatnd$ of $\cscatnd$ whose objects are commutative
$C^*$-algebras. This functor has the properties we expect:

\begin{thm}\label{nondegcommOK}
  The functor $C_0:\topcts\to\comcscatnd$ is an equivalence of
  categories.
\end{thm}

\begin{proof}
  To say that $C_0$ is an equivalence means that there is a functor
  $G:\comcscatnd\to\topcts$ such that $C_0\circ G$ and $G\circ C_0$
  are naturally isomorphic to the identity functors. To verify that it
  is an equivalence, though, it suffices to show that every object in
  $\comcscatnd$ is isomorphic to one of the form $C_0(X)$, which is
  exactly what the Gelfand-Naimark theorem says, and that $C_0$ is a
  bijection on each set $\Mor(X,Y)$ of morphisms (see
  \cite[page~91]{ML}). Injectivity is easy: since $C_0(Y)$ separates
  points of $Y$, $a\circ f=a\circ g$ for all $a\in C_0(Y)$ implies
  that $f(x)=g(x)$ for all $x\in X$. For surjectivity, we suppose that
  $\phi:C_0(Y)\to C_0(X)$ is a nondegenerate homomorphism. Then for
  each $x\in X$, the composition $\epsilon_x\circ\phi$ with the
  evaluation map is a homomorphism from $C_0(Y)$ to $\C$, and the
  non-degeneracy of $\phi$ implies that $\epsilon_x\circ\phi$ is
  non-zero. Since $y\mapsto\epsilon_y$ is a homeomorphism of $Y$ onto
  the maximal ideal space of $C_0(Y)$, there is a unique $f(x)\in Y$
  such that $\epsilon_x\circ\phi=\epsilon_{f(x)}$, and
  $f=\epsilon^{-1}\circ\phi\circ\epsilon$ is continuous. The equation
  $\epsilon_x\circ\phi=\epsilon_{f(x)}$ then says precisely that
  $\phi=C_0(f)$.
\end{proof}

The result in \cite[page~91]{ML} which we have just used is a little
unnerving to analysts. (Well, to us, anyway.) Its proof, for example,
makes carefree use of the axiom of choice. So it is perhaps reassuring
that in the situation of Theorem~\ref{nondegcommOK}, there is a
relatively concrete inverse functor $\Delta$ which takes a commutative
$C^*$-algebra $A$ to its maximal ideal space $\Delta(A)$. (We say
``relatively concrete'' here because the axiom of choice is also used
in the proof that the Gelfand transform is an isomorphism.) The
argument on page 92 of \cite{ML} shows that, once we have chosen
isomorphisms $\eta_A:A\to C(\Delta(A))$ for every commutative
$C^*$-algebra $A$, there is exactly one way to extend $\Delta$ to a
functor in such a way that $\eta:=\{\eta_A:A\in \Obj(\comcscatnd)\}$
is a natural isomorphism. If we choose $\eta_A:A\to C_0(\Delta(A))$ to
be the Gelfand transform, then the functor $\Delta$ takes a morphism
$\phi:A\to M(B)$ to the map
$\Delta(\phi):\omega\to\omega\circ\phi$. So we have the following
naturality result:

\begin{cor}
  The Gelfand transforms $\{\eta_A:A\in \Obj(\comcscatnd)\}$ form a
  natural isomorphism between the identity functor on $\comcscatnd$
  and the composition $C_0\circ\Delta$.
\end{cor}

Of course, modulo the existence of the isomorphisms $\eta_A$, which is
the content of the (highly non-trivial) Gelfand-Naimark theorem, this
result can be easily proved directly: we just need to check that for
every morphism $\phi:A\to M(B)$ the following diagram commutes in
$\comcscatnd$:
\begin{equation*}
  \xymatrix{A\ar[r]^-{\eta_{A}}\ar[d]_{\phi}
    &C_0(\Delta(A))\ar[d]^{C_0(\Delta(\phi))}\\  
    B\ar[r]_-{\eta_{B}}&C_0(\Delta(B)).}
\end{equation*}

\section{Crossed products and Landstad duality}\label{sec:Ld}

Although the category $\cscatnd$ has only been studied in recent
years, possibly for the first time in \cite{kqcat}, nondegenerate
homomorphisms have been around for years. For example, the unitary
representations $U$ of a locally compact group $G$ on a Hilbert space
$H$ are in one-to-one correspondence with the nondegenerate
representations $\pi_U$ of the group algebras $L^1(G)$ or $C^*(G)$ on
$H$.  In this context, ``nondegenerate'' usually means that the
elements $\pi_U(a)h$ span a dense subspace of $H$, but this is
equivalent to the nondegeneracy of $\pi_U$ as a homomorphism into
$B(H)=M(\K(H))$. More generally, if $u:G\to UM(B)$ is a strictly
continuous homomorphism into the unitary group of a multiplier
algebra, then there is a unique nondegenerate homomorphism
$\pi_u:C^*(G)\to M(B)$, called the integrated form of $u$, from which
we can recover $u$ by composing with a canonical unitary
representation $k_G:G\to UM(C^*(G))$. The composition here is taken in
the spirit of the category $\cscatnd$: it is the composition in the
usual sense of the extension of $\pi_u$ to $M(C^*(G))$ with $k_G$. We
say that $k_G$ is universal for unitary representations of $G$.

One application of this circle of ideas which will be particularly
relevant here is the existence of the comultiplication $\delta_G$ on
$C^*(G)$, which is the integrated form of the unitary representation
$k_G\otimes k_G:G\to UM(C^*(G)\otimes C^*(G))$. Thus $\delta_G$ is by
definition a nondegenerate homomorphism of $C^*(G)$ into
$M(C^*(G)\otimes C^*(G))$. Its other crucial property is
coassociativity: $(\delta_G\otimes\id)\circ
\delta_G=(\id\otimes\delta_G)\circ \delta_G$, where the compositions
are interpreted as being those in the category $\cscatnd$.

Now suppose that $\alpha:G\to\Aut A$ is an action of a locally compact
group $G$ on a $C^*$-algebra. Nondegeneracy is then built into the
notion of covariant representation of the system: a covariant
representation $(\pi, u)$ of a dynamical system $(A,G,\alpha)$ in a
multiplier algebra $M(B)$ consists of a nondegenerate homomorphism
$\pi:A\to M(B)$ and a strictly continuous homomorphism $u:G\to UM(B)$
such that $\pi(\alpha_t(a))=u_t\pi(a)u_t^*$. The crossed product is
then, either by definition \cite{R} or by theorem
\cite[2.34--36]{tfb2}, a $C^*$-algebra $A\rtimes_\alpha G$ which is
generated (in a sense made precise in those references) by a universal
covariant representation $(i_A,i_G)$ of $(A,G,\alpha)$ in
$M(A\rtimes_\alpha G)$. Each covariant representation $(\pi,u)$ in
$M(B)$ has an integrated form $\pi\rtimes u$ which is a nondegenerate
homomorphism of $A\rtimes_\alpha G$ into $M(B)$ such that
$\pi=(\pi\rtimes u)\circ i_A$ and $u=(\pi\rtimes u)\circ i_G$.

The crossed product $A\rtimes_\alpha G$ carries a dual coaction
$\hat\alpha$, which is the integrated form of $i_G\otimes k_G:G\to
UM((A\rtimes_\alpha G)\otimes C^*(G))$. This is another nondegenerate
homomorphism, and the crucial coaction identity
$(\hat\alpha\otimes\id)\circ \hat\alpha=(\id\otimes\delta_G)\circ
\hat\alpha$ again has to be interpreted in the category
$\cscatnd$. (Makes you wonder how we ever managed without $\cscatnd$.)

There is another version of the crossed-product construction which can
be more suitable for spatial arguments, and which is particularly
important for the issues we discuss in this paper. For any
representation $\pi:A\to B(H_\pi)$, there is a regular representation
$(\tilde\pi,U)$ of $(A,G,\alpha)$ on $L^2(G,H_\pi)$ such that
$(\tilde\pi(a)h)(r)=\pi(\alpha_r^{-1}(a))(h(r))$ and
$\lambda_sh(r)=h(s^{-1}r)$ for $h\in L^2(G,H_\pi)$. The reduced
crossed product $A\rtimes_{\alpha,r}G$ is the quotient of
$A\rtimes_\alpha G$ which has the property that every
$\tilde\pi\rtimes \lambda$ factors through a representation of
$A\rtimes_{\alpha,r}G$, and then $\tilde\pi\rtimes \lambda$ is
faithful whenever $\pi$ is \cite[\S7.2]{tfb2}. The reduced crossed
product is also generated by a canonical covariant representation
$(i_A^r,i_G^r)$, and the dual coaction $\hat\alpha$ factors through a
coaction $\hat\alpha^n:A\rtimes_{\alpha,r}G\to
M((A\rtimes_{\alpha,r}G)\otimes C^*(G))$ characterised by
\begin{equation}\label{charalphan}
  \overline{\hat\alpha^n}\circ i_A^r(a)=i_A^r(a)\otimes 1\ \text{ and
  } \overline{\hat\alpha^n}\circ i_G^r(s)=i_G^r(s)\otimes k_G(s).
\end{equation} 
This coaction is called the normalisation of $\hat\alpha$, and is in
particular normal in the sense that the canonical map $j_{A\rtimes G}$
of $A\rtimes_{\alpha,r}G$ into
$M((A\rtimes_{\alpha,r}G)\rtimes_{\hat\alpha} G)$ is injective (see
Proposition~A.61 of \cite{enchilada}).
 
Kaliszewski and Quigg's motivation for working in the category
$\cscatnd$ came from the following characterisation of the
$C^*$-algebras which arise as reduced crossed products.

\begin{thm}[Landstad, Kaliszewski-Quigg]\label{LKQcorZ}
  Suppose that $B$ is a $C^*$-algebra and $G$ is a locally compact
  group. Then there is a dynamical system $(A,\alpha,G)$ such that $B$
  is isomorphic to $A\rtimes_{\alpha,r}G$ if and only if there is a
  morphism $\pi:C^*(G)\to UM(B)$ in $\cscatnd$ and a nondegenerate
  (see Remark~\ref{Lnondeg} below) normal coaction $\delta:B\to
  M(B\otimes C^*(G))$ such that
  \begin{equation}\label{piuequiv}
    (\pi\otimes\id)\circ\delta_G=\delta\circ\pi.
  \end{equation}
\end{thm}

In \cite{kqcat} the authors say that this result follows from a
theorem of Landstad \cite{L}, and it is certainly true that most of
the hard work is done by Landstad's result. But we think it is worth
looking at the proof; those who are not interested in the subtleties
of coactions should probably skip to the end of the proof below. We
begin by stating Landstad's theorem in modern terminology.

\begin{thm}[Landstad, 1979] \label{Landoriginal} Suppose that $B$ is a
  $C^*$-algebra and $G$ is a locally compact group. Then there is a
  dynamical system $(A,G,\alpha)$ such that $B$ is isomorphic to
  $A\rtimes_{\alpha,r}G$ if and only if there are a strictly
  continuous homomorphism $u:G\to UM(B)$ and a reduced coaction
  $\delta:B\to M(B\otimes C_r^*(G))$ such that
  \begin{itemize}
  \item[(a)] $\bar\delta(u_s)=u_s\otimes \lambda_s$ for $s\in G$, and

    \smallskip
  \item[(b)] $\delta(A)(1\otimes C_r^*(G))=A\otimes C_r^*(G)$.
  \end{itemize}
\end{thm}

The ``reduced coaction'' appearing in Landstad's theorem is required
to have slightly different properties from the full coactions which we
use elsewhere in this paper, and which are used in \cite{enchilada}
and \cite{kqcat}, for example. A reduced coaction on $B$ is an
injective nondegenerate homomorphism of $B$ into $M(B\otimes
C_r^*(G))$ rather than $M(B\otimes C^*(G))$, and it is required to be
coassociative with respect to the comultiplication $\delta_G^r$ on
$C_r^*(G)$.

\begin{remark}\label{Lnondeg}
  Nowadays, the second condition (b) in Theorem~\ref{Landoriginal} is
  usually absorbed into the assertion that $\delta$ is a
  coaction. Everyone agrees that for $ \delta$ 
  to be a coaction $\delta(A)(1\otimes C_r^*(G))$ must be contained in
  $A\otimes C_r^*(G)$, and Landstad described the requirement of
  equality as ``nondegeneracy'', which in view of our emphasis on
  $\cscatnd$ has turned out to be unfortunate terminology. Coactions
  of amenable or discrete groups are automatically nondegenerate in
  Landstad's sense, and dual coactions are always nondegenerate. We
  therefore follow modern usage and assume that all coactions satisfy
  (b), or its analogue in the case of full coactions. (So (b) can now
  be deleted from Theorem~\ref{Landoriginal} and the word
  ``nondegenerate'' from Theorem~\ref{LKQcorZ}.)
\end{remark}

\begin{proof}[Proof of Theorem~\ref{LKQcorZ}]
  For $B=A\rtimes_{\alpha,r}G$, we take $\delta=\hat\alpha^n$ and
  $\pi=\pi_{i_G^r}$. The second equation in \eqref{charalphan} implies
  that
  \begin{align*}
    \overline{(\pi\otimes\id)\circ\delta_G}(k_G(s))&
    =\overline{(\pi\otimes\id)}(k_G(s)\otimes k_G(s))=i_G^r(s)\otimes
    k_G(s)\\ 
    &=\overline{\hat\alpha^n\circ\pi}(k_G(s))
  \end{align*}
  for all $s\in G$, which implies \eqref{piuequiv}.

  Now suppose that there exist $\pi$ and $\delta$ as described. Then
  we define $u:=\bar\pi\circ k_G^r$, and consider the reduction
  $\delta^r$ of $\delta$, which since $\delta$ is normal is just
  $\delta^r:=(\id\otimes\pi_\lambda)\circ \delta$. Now we compute:
  \begin{align*}
    \overline{\delta^r}(u_s)&= \overline{(\id\otimes\pi_\lambda)\circ
      \delta}(\bar\pi(k_G^r(s)))=
    \overline{\id\otimes\pi_\lambda}\circ \bar{\delta} \circ\bar\pi(k_G^r(s))\\
    &=\overline{\id\otimes\pi_\lambda}\circ
    \overline{\pi\otimes\id}\circ\overline{\delta_G}(k_G^r(s))=
    \overline{\pi\otimes\pi_\lambda} (k_G^r(s)\otimes k_G(s))\\
    &=\bar\pi\circ k_G^r(s)\otimes \lambda_s=u_s\otimes \lambda_s.
  \end{align*}
  Thus $u$ and $\delta^r$ satisfy the hypotheses of Landstad's theorem
  (Theorem~\ref{Landoriginal}), and we can deduce from it that $B$ is
  isomorphic to a reduced crossed product.
\end{proof}

Kaliszewski and Quigg then made two further crucial
observations. First, they recognised that there is a category of
coactions  associated to $\cscat$: the objects in
$\cscoactnd(G)$ consist of a full coaction $\delta$ on a $C^*$-algebra
$B$, and the morphisms from $(B,\delta)$ to $(C,\epsilon)$ are
nondegenerate homomorphisms $\phi:B\to M(C)$ such that
$(\phi\otimes\id)\circ\delta=\epsilon\circ\phi$. Then \eqref{piuequiv}
says that the homomorphism $\pi$ in Corollary~\ref{LKQcorZ} is a
morphism in $\cscoactnd(G)$ from $(C^*(G),\delta_G)$ to
$(B,\delta)$. Second, they knew that for every object $a$ and every
subcategory $D$ in a category $C$ there is a comma category
$a\downarrow D$ in which objects are morphisms $f:a\to x$ in $C$ from
$a$ to objects in $D$, and the morphisms from $(f,x)$ to $(y,g)$ are
morphisms $h:x\to y$ in $D$ such that $h\circ f=g$. Thus Landstad's
theorem identifies the reduced crossed products as the $C^*$-algebras
which can be augmented with a coaction $\delta$ and a homomorphism
$\pi$ to form an object in the comma category
$(C^*(G),\delta_G)\downarrow\cscoactndland(G)$.

The main results in \cite{kqcat} concern crossed-product functors
defined on the category $\csactnd(G)$ whose objects are dynamical
systems $(A,G,\alpha)$ and whose morphisms $\phi:(A,\alpha)\to
(B,\beta)$ are nondegenerate homomorphisms $\phi:A\to M(B)$ such that
$\phi\circ\alpha_s=\beta_s\circ \phi$ for $s\in G$ (where yet again
the compositions are taken in $\cscatnd$). The following theorem is
Theorem~4.1 of \cite{kqcat}.

\begin{thm}[Kaliszewski-Quigg, 2009]\label{thmKQ}
  There is a functor $\CP^r$ from $\csactnd(G)$ to the comma category
  $(C^*(G),\delta_G)\downarrow\cscoactndland(G)$ which takes the
  object $(A,\alpha)$ to $(A\rtimes_{\alpha,r}G,\hat\alpha^r, i_G^r)$,
  and this functor is an equivalence of categories.
\end{thm}

Landstad's theorem, in the form of Theorem~\ref{LKQcorZ}, says that
$\CP^r$ is \emph{essentially surjective}: every object in the comma
category is isomorphic to one of the form
$\CP^r(A,\alpha)=A\rtimes_{\alpha,r}G$. Thus Theorem~\ref{thmKQ} can
be viewed as an extension of Landstad's theorem, and Kaliszewski and
Quigg call it ``categorical Landstad duality for actions''. They also
obtain an analogous result for full crossed products.

\section{Proper actions and Landstad duality for
  coactions}\label{sec:QLd}

Quigg's version of Landstad duality for crossed products by coactions
\cite{Q} is also easy to formulate in categories based on
$\cscatnd$. Suppose that $\delta$ is a coaction of $G$ on $C$, and let
$w_G$ denote the function $s\mapsto k_G(s)$, viewed as a multiplier of
$C_0(G,C^*(G))$. A covariant repesentation of $(C,\delta)$ in a
multiplier algebra $M(B)$ consists of nondegenerate homomorphisms
$\pi:C\to M(B)$ and $\mu:C_0(G)\to M(B)$ such that
\[
(\pi\otimes\id)\circ\delta(c)=\mu\otimes \id(w_G)(\pi(c)\otimes
1)\mu\otimes \id(w_G)^*\ \text{ for $c\in C$,}
\]
where, as should seem usual by now, the composition is interpreted in
$\cscatnd$.  The crossed product $C\rtimes_\delta G$ is generated by a
universal covariant representation $(j_C,j_G)$ in $M(C\rtimes_\delta
G)$, in the sense that products $j_C(c)j_G(f)$ span a dense subspace
of $C\rtimes_\delta G$. The crossed product carries a dual action
$\hat\delta$ such that
$\hat\delta_s(j_C(c)j_G(f))=j_C(c)j_G(\rt_s(f))$, where $\rt$ is
defined by $\rt_s(f)(t)=f(ts)$. Quigg's theorem identifies the
$C^*$-algebras which are isomorphic to crossed products by coactions.

\begin{thm}[Quigg, 1992]\label{QLduality}
  Suppose that $G$ is a locally compact group and $A$ is a
  $C^*$-algebra. There is a system $(C,\delta)$ such that $A$ is
  isomorphic to $C\rtimes_\delta G$ if and only if there are a
  nondegenerate homomorphism $\phi:C_0(G)\to A$ and an action $\alpha$
  of $G$ on $A$ such that $(A,\alpha,\phi)$ is an object in the comma
  category $(C_0(G),\rt)\downarrow \csactnd(G)$.
\end{thm}

When $A=C\rtimes_\delta G$, we can take $\phi:=j_G$ and
$\alpha:=\hat\delta$, and the hard bit is to prove the converse. This
is done in \cite[Theorem~3.3]{Q}. It is then natural to look for a
``categorical Landstad duality for coactions'' which parallels the
results of \cite{kqcat}. However, triples $(A,\alpha,\phi)$ of the
sort appearing in Theorem~\ref{QLduality} had earlier (that is, before
\cite{kqcat}) appeared in important work of Rieffel on proper actions,
and it has proved very worthwhile to follow up this circle of ideas in
Rieffel's context. To explain this, we need to digress a little.

If $\alpha:G\to\Aut A$ is an action of a compact abelian group, then
information about the crossed product can be recovered from the fixed
point algebra $A^{\alpha}$, and, more generally, from the spectral
subspaces
\begin{equation*}
  A^{\alpha}(\omega):=\set{a\in
    A:\alpha_{s}(a)=\overline{\omega(a)}a}\quad\text{for
    $\omega\in\widehat G$.}
\end{equation*}
A fundamental result of Kishimoto and Takai \cite[Theorem~2]{KT} says
that if the spectral subspaces are large in the sense that
$A^{\alpha}(\omega)^{*}A^{\alpha}(\omega)$ is dense in $A^{\alpha}$
for every $\omega\in \widehat G$, then $A\rtimes_{\alpha}G$ is Morita
equivalent to $A^{\alpha}$.  There is as yet no completely
satisfactory notion of a free action of a group on a \cs-algebra (see
\cite{NCP}, for example), but having large spectral subspaces is one
example of such a notion.

When $G$ is locally compact, the fixed-point algebra is often trivial.
For example, if $\rt$ is the action of $G=\Z$ on $\R$ by right
translation, then $f\in C_{0}(\R)^{\rt}$ if and only if $f$ is
periodic with period $1$, which since $f$ vanishes at $\infty$ forces
$f$ to be identically zero.  However, if the orbit space for an action
is nice enough, then the algebra of continuous functions on the orbit
space can be used as a substitute for the fixed-point algebra. A right
action of a locally compact group $G$ on a locally compact space $T$
is called proper if the map $(x,s)\mapsto (x,x\cdot s) :T\times G\to
T\times T$ is proper.  The orbit space $T/G$ for a proper action is
always Hausdorff \cite[Corollary~3.43]{tfb2}, and a classical result
of Green \cite{G} says that if the action of $G$ on $T$ is free and
proper, then $C_{0}(T)\rtimes_{\rt}G$ is Morita equivalent to
$C_{0}(T/G)$ (for this formulation of Green's result see
\cite[Remark~4.12]{tfb2}).  We want to think of $C_{0}(T/G)$ as a
subalgebra of the multiplier algebra $M(C_{0}(G))=\cb(T)$ which is
invariant under the extension $\bar\rt$.

In the past twenty-five years, many researchers have investigated
analogues of free and proper actions for noncommutative \cs-algebras
\cite{RW, proper, exel, meyer, aHRWproper, integrable, aHRWproper2}.
Here we are interested in the notion of proper action $\alpha:G\to\Aut
A$ introduced by Rieffel \cite{proper}. He assumes that there is an
$\alpha$-invariant subalgebra $A_{0}$ of $A$ with properties like
those of the subalgebra $C_{c}(T)$ of $C_{0}(T)$, and that there is an
$M(A)^{\alpha}$-valued inner product on $A_{0}$.  The completion
$Z(A,\alpha)$ of $A_{0}$ in this inner product is a full Hilbert
module over a subalgebra $A^{\alpha}$ of $M(A)^{\alpha}$, which
Rieffel calls the \emph{generalized fixed point algebra} for $\alpha$.
The algebra $\K(Z(A,\alpha))$ of generalized compact operators on
$Z(A,G,\alpha)$ sits naturally as an ideal $E(\alpha)$ in the reduced
crossed product $A\rtimes_{\alpha,r}G$ \cite[Theorem~1.5]{proper}.
The action $\alpha$ is \emph{saturated} when $E(\alpha)$ is all of the
reduced crossed product.  Thus when $\alpha$ is proper and saturated,
$A\rtimes_{\alpha,r}G$ is Morita equivalent to $A^{\alpha}$.
Saturation is a freeness condition: if $G$ acts properly on $T$, then
$\rt:G\to \Aut(C_{0}(T))$ is proper with respect to $C_{c}(G)$, and
the action is saturated if and only if $G$ acts freely \cite[\S3]{MR}.
On the face of it, though, Rieffel's bimodule $Z(A,\alpha)$ and the
fixed-point algebra $A^\alpha$ depend on the choice of subalgebra
$A_{0}$, and it seems unlikely that Rieffel's process is functorial.

The connection with our categories lies in a more recent theorem of
Rieffel which identifies a large family of proper actions for which
there is a canonical choice of the dense subalgebra $A_0$
\cite[Theorem~5.7]{integrable}.

\begin{thm}[Rieffel, 2004]\label{rieffcomma}
  Suppose that a locally compact group $G$ acts freely and properly on
  the right of a locally compact space $T$, and $(A,G,\alpha)$ is a
  dynamical system such that there is a nondegenerate homomorphism
  $\phi:C_0(T)\to M(A)$ satisfying $\phi\circ \rt=\alpha\circ \phi$
  (with composition in the sense of $\cscatnd$). Then $\alpha$ is
  proper and saturated with respect to the subalgebra
  $A_0=\phi(C_c(T))A\phi(C_c(T))$.
\end{thm}

\begin{example}\label{homogex}
  A closed subgroup $H$ of a locally compact group $G$ acts freely and
  properly on $G$, and hence we can apply Theorem~\ref{rieffcomma} to
  the pair $(T,G)=(G,H)$ and to the canonical map $j_G:C_0(G)\to
  M(C\rtimes_\delta G)$. In this case, highly nontrivial results of
  Mansfield \cite{man} can be used to identify the fixed-point algebra
  $(C\rtimes_\delta G)^{\hat\delta}$ with the crossed product
  $C\rtimes_{\delta,r}(G/H)$ by the homogeneous space
  \cite[Remark~3.4]{hrman}. (These crossed products were introduced in
  \cite{ekr}; the relationship with the crossed product
  $C\rtimes_{\delta|} (G/H)$ by the restricted coaction, which makes
  sense when $H$ is normal, is discussed in \cite[Remark~2.2]{ekr}.)
  Then Theorem~3.1 of \cite{hrman} shows that Rieffel's Morita
  equivalence between $(C\rtimes_\delta G)\rtimes_{\hat\delta,r}H$ and
  $(C\rtimes_\delta G)^{\hat\delta}$ extends Mansfield's imprimitivity
  theorem for coactions to arbitary closed subgroups (as opposed to
  the amenable normal subgroups in Mansfield's original theorem
  \cite[Theorem~27]{man} and the normal ones in \cite{kqman}).
\end{example}

From our categorical point of view, the hypotheses on $\phi$ in
Theorem~\ref{rieffcomma} say precisely that
$(A,\alpha,\phi):=((A,\alpha),\phi)$ is an object in the comma
category $(C_0(T),\rt)\downarrow \csactnd(G)$. Then Rieffel's theorem
implies that $(A,\alpha,\phi)\mapsto A^\alpha$ is a construction which
takes objects in the comma category to objects in the category
$\cscatnd$. One naturally asks: is this construction functorial? More
precisely, is there an analogous construction on morphisms which which
makes $(A,\alpha,\phi)\mapsto A^{\alpha}$ into a functor from
$(C_0(T),\rt)\downarrow \csactnd(G)$ to $\cscatnd$?

This question was answered in \cite[\S2]{kqrproper} using a new
construction of Rieffel's generalized fixed-point algebra. The crucial
ingredient is an averaging process $E$ of Olesen and Pedersen
\cite{OP1,OP2}, which was subsequently developed by Quigg in \cite{Q2,
  QR} and used extensively in his proof of
Theorem~\ref{QLduality}. This averaging process $E$ makes sense on the
dense subalgebra $A_0=\phi(C_c(T))A\phi(C_c(T))$, and satisfies
\[
\phi(f)E(\phi(g)a\phi(h))=\int_G \phi(f)\alpha_s(\phi(g)a\phi(h))\,ds\
\text{ for $f,g,h\in C_c(T)$;}
\]
the integral on the right has an unambiguous meaning because
properness implies that $s\mapsto f\rt_s(g)$ has compact support. It
is shown in \cite[Proposition ~2.4]{kqrproper} that the closure of
$E(A_0)$ is a $C^*$-subalgebra of $M(A)$, which we denote by
$\Fix(A,\alpha,\phi)$ to emphasise all the data involved in the
construction. It is shown in \cite[Proposition~3.1]{kqrproper} that
$\Fix(A,\alpha,\phi)$ and Rieffel's $A^\alpha$ are exactly the same
subalgebra of $M(A)$. If $\sigma:(A,\alpha,\phi)\to (B,\beta,\psi)$ is
a morphism in the comma category, so that in particular $\sigma$ is a
nondegenerate homomorphism from $A$ to $M(B)$, then the extension
$\bar\sigma$ maps $\Fix(A,\alpha,\phi)$ into $M(\Fix(B,\beta,\psi))$,
and is nondegenerate. (This is Proposition~2.6 of \cite{kqrproper}; a
gap in the proof of nondegeneracy is filled in Corollary~2.3 of
\cite{aHKRW4}.)

\begin{thm}[Kaliszewski-Quigg-Raeburn, 2008]\label{Fixfunct}
  Suppose that a locally compact group $G$ acts properly on the right
  of a locally compact space $T$. Then the assignments
  $(A,\alpha,\phi)\mapsto \Fix(A,\alpha,\phi)$ and $\sigma\mapsto
  \bar\sigma|_{\Fix(A,\alpha,\phi)}$ form a functor from
  $(C_0(T),\rt)\downarrow \csactnd(G)$ to $\cscatnd$.
\end{thm}

To return to the setting of Quigg-Landstad duality, we take
$(T,G)=(G,G)$ in this theorem. This gives us a functor $\Fix$ from
$(C_0(G),\rt)\downarrow \csactnd(G)$ to $\cscatnd$. Because the
fixed-point algebra $\Fix(A,\alpha,\phi)$ is defined using the same
averaging process $E$ as Quigg used in \cite[\S3]{Q},
$\Fix(A,\alpha,\phi)$ is the same as the algebra $C$ constructed by
Quigg (unfortunately for us, he called it $B$). So Quigg proves in
\cite{Q} that
\[
\delta_A(c)=\phi\otimes\pi_\lambda(w_G)(c\otimes
1)\phi\otimes\pi_\lambda(w_G)^*
\]
defines a reduced coaction of $G$ on $C=\Fix(A,\alpha,\phi)$, and that
$A$ is isomorphic to the crossed product $C\rtimes_{\delta^A} G$. An
examination of the proof of \cite[Theorem~4.7]{Q2} shows that the
similar formula
\[
\delta_A^f(c)=\phi\otimes\id(w_G)(c\otimes 1)\phi\otimes\id(w_G)^*
\]
defines the unique full coaction with reduction $\delta_A$. The
argument on page ~2960 of \cite{kqrproper} shows that this
construction respects morphisms, so that $\Fix$ extends to a functor
$\Fix_G$ from $(C_0(G),\rt)\downarrow \csactnd(G)$ to
$\cscoactndland(G)$. The following very satisfactory ``categorical
Landstad duality for coactions'' is Corollary~4.3 of \cite{kqrproper}.

\begin{thm}[Kaliszewski-Quigg-Raeburn, 2008]
  Let $G$ be a locally compact group. Then $(C,\delta)\mapsto
  (C\rtimes_\delta G,\hat\delta,j_G)$ and $\pi\mapsto \pi\rtimes \id$
  form a functor from $\cscoactndland(G)$ to $(C_0(G),\rt)\downarrow
  \csactnd(G)$. This functor is an equivalence of categories with
  quasi-inverse $\Fix_G$.
\end{thm}

In fact, this is a much more satisfying theorem than its analogue for
actions because we have a specific construction of a quasi-inverse. We
would be interested to see an analogous process for $\Fix$ing over
coactions.

\section{Naturality}
\label{sec:naturality}

Now that we have a functorial version $\Fix$ of Rieffel's generalised fixed-point algebra, we remember that the main point of Rieffel's paper \cite{proper} was to construct a Morita equivalence between $A^\alpha=\Fix(A,\alpha,\phi)$ and the reduced crossed product $A\rtimes_{\alpha, r}G=\RCP(A,\alpha,\phi)$. This equivalence is implemented by an $(A\rtimes_{\alpha, r}G)$\,--\,$\Fix(A,\alpha,\phi)$ imprimitivity bimodule $Z(A,\alpha,\phi)$. There is another category $\cscat$ of $C^*$-algebras in which the isomorphisms are given by imprimitivity bimodules, so it makes sense to ask whether these isomorphisms are natural.
Of course, before discussing this problem, we need to be clear about what the category $\cscat$ is.

If $A$ and $B$ are $C^*$-algebras, then a right-Hilbert $A$\,--\,$B$ bimodule  is a right Hilbert $B$-module $X$ which is also a left
$A$-module via a nondegenerate homomorphism of $A$ into the
algebra $\L(X)$ of bounded adjointable operators on $X$. (These are sometimes called $A$\,--\,$B$ correspondences.) The objects in $\cscat$ are $C^*$-algebras, and the morphisms from $A$ to $B$ are the isomorphism classes $[X]$ of full right-Hilbert $A$\,--\,$B$ bimodules. Every nondegenerate
homomorphism $\phi:A\to M(B)$ gives a right-Hilbert bimodule: view  $B$ as a right Hilbert $B$-module over itself with $\rip
  B<b_{1},b_{2}>:= b_{1}^{*}b_{2}$, and define the action of $A$ by $a\cdot b:=\phi(a)b$. We denote the
  isomorphism class of this bimodule by
  $[\phi]$. In \cite{taco}, it is shown that $[\phi]=[\psi]$ if and only if there exists $u\in UM(B)$ such that $\psi=(\Ad u)\circ \phi$, so we are not just adding more morphisms to $\cscatnd$, we are also slightly changing the morphisms we already have.
    
If $_{A}X_{B}$ and $_{B}Y_{C}$ are right Hilbert bimodules, then we
define the composition using the internal tensor product: $[Y][X]:=[X\tensor_{B}Y].$ The identity morphism $1_A$ on $A$ is $[{}_{A}A_{A}]=[\id_A]$. Now we can see why we have had to take isomorphism
classes of bimodules as our morphisms: the bimodule $A\tensor_{A}X$ representing $1_A[X]=[X][\id_A]$ is only isomorphic to
$X$.  A similar subtlety arises when checking that composition of
morphisms is associative. The details are in \cite[Proposition~2.4]{taco}. In \cite[Proposition~2.6]{taco}, it is shown that the isomorphisms from $A$ to $B$  in $\cscat$ are the classes $[X]$ in which $X$ is an imprimitivity bimodule, so that $X$ also carries a left inner product $\lip A<x,y>$ such that $\lip A<x,y>\cdot z=x\cdot\rip B<y,z>$. Similar results were obtained independently by Landsman \cite{landsman, landsman2} and by Schweizer \cite{S}, and a slightly more general category in which the bimodules are not required to be full as right Hilbert modules was considered in \cite{enchilada}.

Theorem~3.2 of \cite{kqrproper} says that, for every nondegenerate homomorphism $\sigma:A\to M(B)$, the diagram
\begin{equation}
  \label{eq:2}
  \xymatrix@C=8pc@R=3pc{A\rtimes_{\alpha,r}G\ar[r]^-{[Z(A,\alpha,\phi)]}
    \ar[d]_{[\sigma\rtimes\id]} & \Fix(A,\alpha,\phi)\ar[d]^{[\sigma|]} \\
    B\rtimes_{\beta,r}G\ar[r]_-{[Z(B,\beta,\psi)]}&\Fix(B,\beta,\psi)}
\end{equation}
commutes in $\cscat$, which means that 
\[
Z(A,\alpha,\phi)\otimes_{\Fix(A,\alpha,\phi)}\Fix(B,\beta,\psi)\ \text{ and }\ 
(B\rtimes_{\beta,r}G)\otimes_{B\rtimes_{\beta,r}G}Z(B,\beta,\psi)
\]
are isomorphic as right-Hilbert $(A\rtimes_{\alpha,r}G)$\,--\,$\Fix(B,\beta,\psi)$ bimodules. Thus Rieffel's bimodules (or rather, the morphisms in $\cscat$ which they determine) implement a natural isomorphism between the functors $\RCP$ and $\Fix$ from $(C_0(T),\rt)\downarrow \csactnd(G)$ to $\cscat$.

This naturality theorem certainly has interesting applications to nonabelian duality, where it gives naturality for the extension in \cite{hrman} of Mansfield's  imprimitity theorem to closed subgroups  (see \cite[Theorem~6.2]{kqrproper}). However, it is slightly unsatisfactory in that the functors involved go from a category built from $\cscatnd$ to $\cscat$: we were forced to go into $\cscat$ because the bimodules $Z$ do not define morphisms in $\cscatnd$, but in the diagram \eqref{eq:2} we have not fully committed to the move. Our goal in \cite{aHKRWproper-n} was to find versions of the same functors defined on a category built from $\cscat$ --- that is, ones in which the morphisms are implemented by bimodules --- to establish that Rieffel's Morita equivalence gives a natural isomorphism between these functors, and to apply the results to nonabelian duality. We will describe our progress in the next section.

\section{Upgrading to \cscat}
\label{sec:upgrading-cscat-its}

Proposition~3.3 of \cite{taco} says that for every locally compact
group $G$, there is a category $\csact(G)$ whose objects are dynamical
systems $(A,\alpha)=(A,G,\alpha)$ and whose morphisms are obtained by
adding actions to the morphisms of $\cscat$. Formally, if $(A,\alpha)$
and $(B,\beta)$ are objects in $\cscat(G)$ and ${}_AX_B$ is a
right-Hilbert bimodule, then an action of $G$ on a right-Hilbert
bimodule $X$ is a strongly continuous homomorphism of $G$ into the
linear isomorphisms of $X$ such that
\begin{equation*}
  u_{s}(a\cdot x\cdot b)=\alpha_{s}(a)\cdot u_{s}(x)\cdot
  \beta_{s}(b)\quad\text{and}\quad \brip B<u_{s}(x),u_{s}(y)>
  =\beta_{s} \bigl(\rip B<x,y>\bigr), 
\end{equation*}
and the morphisms in $\cscat(G)$ are isomorphism classes of pairs
$(X,u)$.

Next we consider a free and proper action of $G$ on a locally compact
space $T$ and look for an analogue of the comma category for the
system $(C_0(T),\rt)$. The objects are easy: to ensure that $\Fix$ is
defined on objects, we need to insist that every system $(A,\alpha)$
is equipped with a nondegenerate homomorphism $\phi:C_{0}(T)\to M(A)$
which is $\rt$\,--\,$\alpha$ equivariant.  We choose to use the
\emph{semi-comma category}
$\csact\bigl(G,\bigl(C_{0}(T),\rt\bigr)\bigr)$ in which the objects
are triples $(A,\alpha,\phi)$, and the morphisms from
$(A,\alpha,\phi)$ to $(B,\beta,\psi)$ are just the morphisms from
$(A,\alpha)$ to $(B,\beta)$ in $\csact(G)$. In
\cite[Remark~2.4]{aHKRWproper-n} we have discussed our reasons for
adding the maps $\phi$ to our objects and then ignoring them in our
morphisms, and the discussion below of how we $\Fix$ morphisms should
help convince sceptics that this is appropriate.

We know how to $\Fix$ objects in the semi-comma category $\csactgct$,
and we need to describe how to $\Fix$ a morphism $[(X,u)]$ from
$(A,\alpha,\phi)$ to $(B,\beta,\psi)$. We begin by factoring the
morphism $[X]$ in $\cscat$ as the composition
$[{}_{\K(X)}X_B][\kappa_A]$ of the isomorphism associated to the
imprimitivity bimodule ${}_{\K(X)}X_B$ with the morphism coming from
the nondegenerate homomorphism $\kappa_A:A\to M(\K(X))=\L(X)$
describing the left action of $A$ on $X$ (see Proposition~2.27 of
\cite{enchilada}). The action $u$ of $G$ on $X$ gives an action $\mu$
of $G$ on $\K(X)$ such that
$\mu_s(\Theta_{x,y})=\Theta_{u_s(x),u_s(y)}$, and then $\kappa_A$
satisfies $\kappa_A\circ \alpha_s=\mu_s\circ\kappa_A$. So the morphism
$[{}_{(A,\alpha)}(X,u)_{(B,\beta)}]$ in $\cscat(G)$ factors as
$[{}_{(\K(X),\mu)}(X,u)_{(B,\beta)}][\kappa_A]$. Now $\kappa_A$ is a
morphism in $\cscatnd(G)$ from $(A,\alpha, \phi)$ to
$(\K(X),\mu,\kappa_A\circ\phi)$, and hence by Theorem~\ref{Fixfunct}
restricts to a morphism $\kappa_A|$ from $\Fix(A,\alpha, \phi)$ to
$\Fix(\K(X),\mu,\kappa_A\circ\phi)$. We want to define $\Fix$ so that
it is a functor, so our definition must satisfy
\begin{equation}\label{fixcomp}
  \Fix([(X,u)])=\Fix([{}_{(\K(X),\mu)}(X,u)_{(B,\beta)}])\Fix([\kappa_A]).
\end{equation}
Since we don't want to change the meaning of $\Fix$ on morphisms in
$\cscatnd$, our strategy is to define $\Fix([\kappa_A]):=[\kappa_A|]$,
figure out how to $\Fix$ imprimitivity bimodules, and then use
\eqref{fixcomp} to define $\Fix([(X,u)])$.

So we suppose that $(A,\alpha,\phi)$ to $(B,\beta,\psi)$ are objects
in the semi-comma category $\csactgct$, and that $[(X,u)]$ is an
equivariant $(A,\alpha)$\,--\,$(B,\beta)$ imprimitivity bimodule. We
emphasise that, because of our choice of morphisms in $\csactgct$, we
do not make any assumption relating the actions of $\phi$ and $\psi$
on $X$.  We let $\widetilde X:=\{\flat(x):x\in X\}$ be the dual
bimodule, and form the linking algebra
\begin{equation*}
  L(X):=
  \begin{pmatrix}
    A&X\\\widetilde X&B
  \end{pmatrix},
\end{equation*}
as in the discussion following \cite[Theorem~3.19]{tfb}. Then
\begin{equation*}
  L(u):=
  \begin{pmatrix}
    \alpha& u\\ \flat(u) & \beta
  \end{pmatrix}
  \quad\text{and}\quad
  \phi_{L}:= \begin{pmatrix}
    \phi&0\\0&\psi
  \end{pmatrix}
\end{equation*}
define an action $L(u)$ of $G$ on $L(X)$ and a nondegenerate
homomorphism $\phi_{L}$ of $C_{0}(T)$ into $M(L(X))$ which intertwines
$\rt$ and $L(u)$.  Then $(L(X),L(u),\phi_{L})$ is an object in
$\csactgct$, and (reverting to Rieffel's notation to simplify the
formulas) we can form $L(X)^{L(u)}:=\Fix(L(X),L(u),\phi_{L})$. It
follows quite easily from the construction of $\Fix$ in
\cite[\S2]{kqrproper} that the diagonal corners in $L(X)^{L(u)}$ are
$A^\alpha$ and $B^{\beta}$, and we \emph{define} $X^u$ to be the upper
right-hand corner, so that
\begin{equation*}
  L(X)^{L(u)}=
  \begin{pmatrix}
    A^{\alpha}&X^{u}\\ * & B^{\beta}
  \end{pmatrix};
\end{equation*}
with the actions and inner products coming from the operations in
$L(X)^{L(u)}$, $X^{u}$ becomes an
$A^{\alpha}$\,--\,$B^{\beta}$-imprimitivity bimodule (see
\cite[Proposition~3.1]{tfb}). We now define $\Fix([X,u]):=X^u$, and
use \eqref{fixcomp} to define $\Fix$ in general, as described above.

With this definition, Theorem~3.3 of
\cite{aHKRWproper-n} says:

\begin{thm}
  \label{thm-3.3}
  Suppose that $T$ is a free and proper right $G$-space.  Then the
  assignments
  \begin{equation*}
    (A,\alpha,\phi)\mapsto \Fix(A,\alpha,\phi)\quad\text{and}\quad
    [(X,u)]\mapsto \Fix([X,u])
  \end{equation*}
  form a functor $\Fix$ from the semi-comma category $\csactgct$ to
  $\cscat$.
\end{thm}

Proving that $\Fix$ preserves the composition of morphisms is
surprisingly complicated, and involves several non-trivial steps. For
example, we needed to show that if $_{(A,\alpha)}(X,u)_{(B,\beta)}$
and $_{(B,\beta)}Y_{(C,\gamma)}$ are imprimitivity bimodules
implementing isomorphisms in $\csactgct$, then
$(X\tensor_{B}Y)^{u\tensor v}$ is isomorphic to
$X^{u}\tensor_{B^{\beta}}Y^{v}$ as $A^{\alpha}$\,--\,$C^{\gamma}$
imprimitivity bimodules.

It follows from \cite[Theorem~3.7]{enchilada} that $\RCP$ is a functor
from $\csactgct$ to $\cscat$ which takes a morphism $[X,u]$ to the
class of the Combes bimodule $[X\rtimes_{u,r}G]$. We can now state the
main naturality result, which is Theorem~3.5 of \cite{aHKRWproper-n}.

\begin{thm}
  \label{thm-3.5}
  Suppose that a locally compact group $G$ acts freely and properly on
  a locall compact space $T$. Then the Morita equivalences
  $Z(A,\alpha,\phi)$ form a natural isomorphism between the functors
  $\RCP$ and $\Fix$ from $\csactgct$ to $\cscat$.
\end{thm}

The proof of Theorem~\ref{thm-3.5} relies on factoring morphisms: then
Theorem~3.2 of \cite{kqrproper} gives the result for the nondegenerate
homomorphism, and standard linking algebra techniques give the other
half.

We saw in Example~\ref{homogex} that Rieffel's Morita equivalence can
be used to generalise Mansfield's imprimitivity theorem to crossed
products by homogeneous spaces, and we want to deduce from
Theorem~\ref{thm-3.5} that this imprimitivity theorem gives a natural
isomorphism. To get the imprimitivity theorem in
Example~\ref{homogex}, we applied Rieffel's Theorem~\ref{rieffcomma}
to a crossed product $C\rtimes_\delta G$. So the naturality result we
seek relates the compositions of $\RCP$ and $\Fix$ with a
crossed-product functor.

Suppose as in Example~\ref{homogex} that $H$ is a closed subgroup of a
locally compact group $G$. We know from Theorem~2.15 of
\cite{enchilada} that there is a category $\cscoactn(G)$ whose objects
are normal coactions $(B,\delta)$, and whose morphisms are isomorphism
classes of suitably equivariant right-Hilbert bimodules. We also know
from Theorem~3.13 of \cite{enchilada} that there is a functor
$\CP:\cscoactn(G)\to \csact(H)$, and adding the canonical map $j_G$
makes $\CP$ into a functor with values in the comma category
$(C_0(G),\rt)\downarrow \csact(H)$. We show in
\cite[Proposition~5.5]{aHKRWproper-n} that there is a functor $\RCPgh$
which sends $(B,\delta)$ to the crossed product
$B\rtimes_{\delta,r}(G/H)$ by the homogeneous space $G/H$, and that
this functor coincides with $\Fix\circ \CP$. We saw in
Example~\ref{homogex} that Rieffel's bimodules
$Z(B\rtimes_{\delta}G,\hat\delta|H,j_{G})$ implement an Morita
equivalence between $(B\rtimes_{\delta}G)\rtimes_{\hat\delta|}H$ and
$B\rtimes_{\delta,r}G/H$.  Write $\RCPh$ for the functor from
$\csact(G)$ to $\cscat$ sending $(C,\gamma)\mapsto
C\rtimes_{\gamma,r}H$.  Then the general naturality result above gives
the following theorem, which is Theorem~5.6 of~\cite{aHKRWproper-n}.

\begin{cor}
  \label{thm-5.6}
  Let $H$ be a closed subgroup of $G$. Then Rieffel's Morita
  equivalences $Z(G\rtimes_{\delta}G,\hat\delta|H,j_{G})$ implement a
  natural isomorphism between the functors $\RCPh\circ\CP$ and
  $\RCPgh$ from $\cscoactn(G)$ to $\cscat$.
\end{cor}

Corollary~\ref{thm-5.6} extends Theorem~4.3 of \cite{enchilada} to
non-normal subgroups, and extends Theorem~6.2 of \cite{kqrproper} to
categories based on $\cscat$ rather than ones based on $\cscatnd$.

\section{Induction-in-stages and fixing-in-stages}
\label{sec:appl-nonc-dual}

Rieffel's theory of proper actions seems to be a powerful tool for
studying systems in the comma or semi-comma category associated to a
pair $(T,G)$. Corollary~\ref{thm-5.6} is, we think, an impressive
first example. As another example, we discuss an approach to
induction-in-stages which works through the same general machinery,
and which we carried out in \cite{aHKRW4}.

The original purpose of an imprimitivity theorem was to provide a way
of recognising induced representations (as in, for example,
\cite{mackey}), and Rieffel's theory of Morita equivalence for
$C^*$-algebras was developed to put imprimitivity theorems in a
$C^*$-algebraic context \cite{rieff, rieff2}. One can reverse the
process: a Morita equivalence $X$ between a crossed product
$C\rtimes_{\alpha} G$ and another $C^*$-algebra $B$ gives an induction
process $X\dashind$ which takes a representation of $B$ on $\HH$ to a
representation of $C$ on $X\otimes_B\HH$, and for which there is a
ready-made imprimitivity theorem (see, for example,
\cite[Proposition~2.1]{aHKRW-JOT}). The situation is slightly less
satisfactory when one has a reduced crossed product, but one can still
construct induced representations and prove an imprimitivity theorem.

Mansfield's imprimitivity theorem, as extended to homogeneous spaces
in \cite{hrman}, gives an induction process $\Ind_{G/H}^G$ from
$B\rtimes_{\delta,r}(G/H)$ to $B\rtimes_\delta G$ which comes with an
imprimitivity theorem. One then asks whether this induction process
has the other properties which one would expect. For example, we ask
whether we can induct-in-stages: if we have subgroups $H$, $K$ and $L$
with $H\subset K\subset L$, is $\Ind_{G/K}^{G/H}(\Ind_{G/L}^{G/K}\pi)$
unitarily equivalent to $\Ind_{G/L}^{G/H}\pi$? If the subgroups are
normal and amenable, then the induction processes are those defined by
Mansfield \cite{man}, and induction-in-stages was established in
\cite[Theorem~3.1]{kqrold}. For non-normal subgroups, not much seems
to be known. There are clearly issues: for example, the subgroups $H$
and $K$ have to be normal in $L$ for the three induction processes to
be defined.

We tackled this problem in \cite{aHKRW4} using our semi-comma
category. Suppose that $(T,G)$ is as usual, $N$ is a closed normal
subgroup of $G$, and $(A,\alpha,\phi)$ is an object in $\csact(G,
(C_0(T),\rt))$. Then $N$ also acts freely and properly on $T$, so we
can form the fixed-point algebra $\Fix_N(A,\alpha|_N,\phi)$. The
quotient $G/N$ has a natural action $\alpha^{G/N}$ on
$A^{\alpha|N}:=\Fix(A,\alpha|_N,\phi)$, and the map $\phi$ induces a
homomorphism $\phi_N:C_0(T/N)\to M(A^{\alpha|N})$ such that
$(A^{\alpha|N},\alpha^{G/N},\phi_N)$ is an object in the semi-comma
category $\csact(G/N,(C_0(T/N),\rt))$. We prove in \cite{aHKRW4} that
$\Fix_N$ extends to a functor
\[
\Fix^{G/N}_N:\csact(G, C_0(T),\rt)\to\csact(G/N, (C_0(T/N),\rt)),
\]
and that the functors $\Fix_{G/N}\circ\Fix_N^{G/N}$ and $\Fix_G$ are
naturally isomorphic (see \cite[Theorem~4.5]{aHKRW4}). The first
difficulty in the proof is showing that the functor $\Fix_N$ has an
equivariant version: because the functor $\Fix$ is defined using the
factorisation of morphisms, we have to track carefully through the
constructions in \cite{aHKRWproper-n} to make sure that they all
respect the actions of $G/N$.

Applying this result on ``fixing-in-stages'' with $(T,G)=(L/H,K/H)$,
gives the following version of induction-in-stages, which is
Theorem~7.3 of \cite{aHKRW4}.

\begin{thm}
  \label{thm-7.3}
  Suppose that $\delta$ is a normal coaction of $G$ on $B$, and that
  $H$, $K$ and $L$ are closed subgroups of $G$ such that $H\subset K
  \subset L$ with both $H$ and $K$ normal in $L$.  Then for every
  representation $\pi$ of $B\rtimes_{\delta,r}(G/L)$, the
  representation $\Ind_{G/K}^{G/H}( \Ind_{G/L}^{G/K}\pi)$ is unitarily
  equivalent to $\Ind_{G/L}^{G/H}\pi$.
\end{thm}

Obviously this is not the last word on the subject, and the normality
hypotheses on subgroups are irritating. However, Mansfield's induction
process is notoriously hard to work with, and it seems remarkable that
one can prove very much at all about an induction process which is
substantially more general than his. We think that Rieffel's theory of
proper actions is proving to be a remarkably malleable and powerful
tool.

\end{document}